\documentclass[10pt,a4paper,twoside]{article}
\usepackage{amsmath,amsfonts,amsthm,amsopn,color,amssymb,enumitem}
\usepackage{palatino}
\usepackage{graphicx}
\usepackage[colorlinks=true]{hyperref}
\hypersetup{urlcolor=blue, citecolor=red, linkcolor=blue}
\usepackage{mathrsfs}

\newcommand{\eps}{\varepsilon}

\newcommand{\R}{\mathbb{R}}

\newcommand{\RT}{{\mathbb{R}^3}}

\renewcommand{\le}{\leqslant}
\renewcommand{\ge}{\geqslant}

\newcommand{\g }{\gamma }

\renewcommand{\l }{\lambda}
\newcommand{\n }{\nabla }

\renewcommand{\O}{\Omega}

\newcommand{\G}{\Gamma}

\newcommand{\BR}{B_R}

\renewcommand{\H}{H^1_0(\Omega)}

\renewcommand{\o}{\omega}

\def\bbm[#1]{\mbox{\boldmath $#1$}}
\newcommand{\beq }{\begin{equation}}
\newcommand{\eeq }{\end{equation}}

\newcommand{\iO}{\int_{\Omega}}

\newtheorem{theorem}{Theorem}[section]
\newtheorem{lemma}[theorem]{Lemma}

\newtheorem{proposition}[theorem]{Proposition}
\newtheorem{remark}[theorem]{Remark}

\title{
On a system involving a critically growing nonlinearity 
\footnote{The authors
are supported by GNAMPA Project ``Problemi ellittici con termini
non locali''.}}

\author{Antonio Azzollini \thanks{Dipartimento di Matematica ed Informatica, Universit\`a degli
Studi della Basilicata,  Via dell'Ateneo Lucano 10, I-85100 Potenza,
Italy, e-mail: {\tt antonio.azzollini@unibas.it}}
 \; \& \;
 Pietro d'Avenia\thanks{Dipartimento di Matematica, Politecnico di
Bari, Via E. Orabona 4, I-70125 Bari, Italy, e-mail: {\tt
p.davenia@poliba.it}}}

\date{}

\begin{document}

\maketitle

\begin{abstract}
This paper deals with the system
\[
\left\{
\begin{array}{ll}
    -\Delta u = \lambda u + q |u|^3 u \phi
    &
    \hbox{in } B_R,\\
    -\Delta \phi=q |u|^5
    &
    \hbox{in } B_R,\\
    u=\phi=0
    &
    \hbox{on } \partial B_R.
\end{array}
\right.
\]
We prove existence and nonexistence results depending on the value of $\lambda$.
\end{abstract}

\noindent{\small \textit{Keywords: Critical nonlinearity, Schr\"odinger-Poisson system}.
}

\section{Introduction}

Recently, in the paper \cite{ADL}, it has been studied the following
system
    \begin{equation}\label{eq:SP}
\left\{
\begin{array}{ll}
-\Delta u=\eta|u|^{p-1}u + \eps q\phi f(u)& \text{in } \O,  \\
- \Delta \phi=2 qF(u)& \text{in } \O,   \\
u=\phi=0 & \text{on }\partial \O,
\end{array}
\right.
    \end{equation}
where $\O \subset \mathbb{R}^3$ is a bounded domain with smooth
boundary $\partial \O$, $1 < p < 5$, $q >0$, $\eps,\eta = \pm 1$,
$f:\R\to\R$ is a continuous function
and $F(s)=\int_0^s f(t)\, dt.$\\
If $f(s)=s$, \eqref{eq:SP} becomes the well known
Schr\"odinger-Poisson system in a bounded domain which has been
investigated by many authors (see e.g. \cite{BF, CS, PS0, PS, RS, S}). In \cite{ADL} it has been
showed that if $f(s)$ grows at infinity as $s^4$, then a variational
approach based on the reduction method (e.g. as in \cite{BF}) becomes more difficult
because of a loss of compactness in the coupling term. In this case
problem \eqref{eq:SP} recalls, at least formally, the more known
Dirichlet problem
    \begin{equation}\label{eq:BN}
    \left\{
\begin{array}{ll}
        -\Delta u =\l u^p + u^5& \text{in } \O,  \\
        u>0& \text{in } \O\\
u=0& \text{in } \partial\O
\end{array}
\right.
    \end{equation}
which has been studied and solved for $p\in [1,5[$ by Brezis and
Nirenberg in the very celebrated paper \cite{BN}. In that paper, by
means of a deep investigation of the compactness property of
minimizing sequences of a suitable constrained functional, it was
showed that, if $p=1$ and $\O$ is a ball, the Palais-Smale condition holds at the level
of the infimum if and only if the parameter $\l$ lies into an interval
depending on the first eigenvalue of the operator $-\Delta$. In the same
spirit of \cite{BN} and \cite{ADL}, in this paper we are interested
in studying the following problem
\begin{equation} \label{P}\tag{$\mathcal P$}
\left\{
\begin{array}{ll}
    -\Delta u = \l u + q |u|^3 u \phi
    &
    \hbox{in } \BR,\\
    -\Delta \phi=q |u|^5
    &
    \hbox{in } \BR,\\
    u=\phi=0
    &
    \hbox{on } \partial \BR.
\end{array}
\right.
\end{equation}
where $\l\in\R$ and $\BR$ is the ball in $\RT$ centered in $0$ with
radius $R$.

As it is well known, problem \eqref{P} is equivalent to that of
finding critical points of a functional depending only on the
variable $u$ and which includes a nonlocal nonlinear term. Many
papers treated functionals presenting both a critically growing
nonlinearity and a nonlocal nonlinearity  (see \cite{AP, CCM, C, ZZ}),
but, up to our knowledge, it has been never
considered the case when the term presenting a critical growth
corresponds with the one containing the nonlocal nonlinearity.\\
From a technical point of view, the use of an approach similar to
that of Brezis and Nirenberg requires different estimates with
respect to those used in the above mentioned papers. Indeed, since
it is just the nonlocal term of the functional the
cause of the lack of compactness, it seems natural to compare it
with the critical Lebesgue norm.\\
The main result we present is the following.
    \begin{theorem}\label{th1}
        Set $\l_1$ the first eigenvalue of $-\Delta$ in $B_R.$
        If $\l \in ]\frac 3{10
        } \l_1,\l_1[$, then
        problem \eqref{P} has
        a positive ground state solution for any $q>0.$
    \end{theorem}
The analogy with the problem \eqref{eq:BN} applies also to some non
existence results. Indeed a classical Poho\u{z}aev obstruction holds for
\eqref{P} according to the following result.
    \begin{theorem}\label{th2}
        Problem \eqref{P} has no nontrivial solution if $\l\le 0$.
    \end{theorem}
Actually, Theorem \ref{th2} holds also if the domain is a general smooth and
star shaped open bounded set.
Moreover, a standard argument allows us also to prove that there
exists no solution to \eqref{P} if $\l\ge\l_1$ (see \cite[Remark 1.1]{BN}).
\\
It remains an open problem what happens if $\l\in ]0,\frac 3{10}
\l_1].$

The paper is so organized:  Section \ref{sec:nonex} is devoted to prove the
nonexistence result which does not require any variational argument;
in Section \ref{sec:ex} we introduce our
variational approach and prove the existence of a positive ground state
solution.

\section{Nonexistence result}\label{sec:nonex}
In this section, following \cite{DM}, we adapt the Poho\u{z}aev arguments in \cite{P} to
our situation.

Let $\O \subset \RT$ be a star shaped domain and $(u,\phi)\in\H\times\H$ be a nontrivial
solution of (\ref{P}). If we multiply the first equation of (\ref{P}) by $x\cdot\n u$ and the second one by $x\cdot\n\phi$ we have that
\begin{align*}
0=&(\Delta u + \l u + q \phi |u|^3 u)(x\cdot\n u)\\
 =&\operatorname{div}\left[(\n u)(x\cdot\n u)\right] - |\n u|^2
   - x\cdot \n \left(\frac{|\n u|^2}{2}\right) + \frac{\l}{2} \n (u^2)
   + \frac{q}{5} x\cdot\n\left(\phi |u|^5\right) - \frac{q}{5} (x\cdot\n\phi)|u|^5\\
 =&\operatorname{div}\left[(\n u)(x\cdot\n u) - x \frac{|\n u|^2}{2} + \frac{\l}{2} x u^2
   +\frac{q}{5} x \phi |u|^5 \right] + \frac{1}{2} |\n u|^2 - \frac{3}{2} \l u^2
   - \frac{3}{5} q \phi |u|^5 - \frac{q}{5} (x\cdot\n\phi) |u|^5
\end{align*}
and
\begin{align*}
0=&(\Delta \phi + q |u|^5)(x\cdot\n \phi)\\
 =&\operatorname{div}\left[(\n \phi)(x\cdot\n \phi)\right] - |\n \phi|^2
   - x\cdot \n \left(\frac{|\n \phi|^2}{2}\right) + q (x\cdot\n\phi) |u|^5\\
 =&\operatorname{div}\left[(\n \phi)(x\cdot\n \phi) - x \frac{|\n \phi|^2}{2}
   \right] + \frac{1}{2} |\n \phi|^2 + q (x\cdot\n\phi) |u|^5.
\end{align*}
Integrating on $\O$, by boundary conditions, we obtain
\begin{equation}\label{eq:Poho1}
-\frac{1}{2} \|\n u \|_2^2 - \frac{1}{2} \int_{\partial\O} \left| \frac{\partial u}{\partial {\bf n}}\right|^2 x\cdot {\bf n} = - \frac{3}{2} \l \|u\|_2^2 -\frac{3}{5} q \iO \phi |u|^5 -\frac{q}{5} \iO (x\cdot\n\phi) |u|^5
\end{equation}
and
\begin{equation}\label{eq:Poho2}
-\frac{1}{2} \|\n\phi\|_2^2 -\frac{1}{2} \int_{\partial\O} \left| \frac{\partial \phi}{\partial {\bf n}}\right|^2 x\cdot {\bf n}
= q \iO (x \cdot \n \phi) |u|^5.
\end{equation}
Substituting \eqref{eq:Poho2} into \eqref{eq:Poho1} we have
\begin{equation}\label{eq:Pohoc}
-\frac{1}{2} \|\n u \|_2^2 - \frac{1}{2} \int_{\partial\O} \left| \frac{\partial u}{\partial {\bf n}}\right|^2 x\cdot {\bf n} = - \frac{3}{2} \l \|u\|_2^2  -\frac{3}{5} q \iO \phi |u|^5 + \frac{1}{10}  \|\n\phi\|_2^2 + \frac{1}{10}  \int_{\partial\O} \left| \frac{\partial \phi}{\partial {\bf n}}\right|^2 x\cdot {\bf n} .
\end{equation}
Moreover, multiplying the first equation of (\ref{P}) by $u$ and the second one by $\phi$ we get
\begin{equation}\label{eq:Ne1}
\|\n u \|_2^2= \l  \| u \|_2^2 + q \iO \phi |u|^5
\end{equation}
and
\begin{equation}\label{eq:Ne2}
\|\n\phi\|_2^2 = q \iO \phi |u|^5.
\end{equation}
Hence, combining \eqref{eq:Pohoc}, \eqref{eq:Ne1} and \eqref{eq:Ne2}, we have
\[
-\l \| u \|_2^2 + \frac{1}{2} \int_{\partial\O} \left| \frac{\partial u}{\partial {\bf n}}\right|^2 x\cdot {\bf n}  + \frac{1}{10}  \int_{\partial\O} \left| \frac{\partial \phi}{\partial {\bf n}}\right|^2 x\cdot {\bf n} = 0
\]
Then, if $\l<0$ we get a contradiction.\\
If $\l=0$, then
\[
 \int_{\partial\O} \left| \frac{\partial \phi}{\partial {\bf n}}\right|^2 x\cdot {\bf n} = 0
\]
and so by the second equation of \eqref{P} we get $\|u\|_5=0$. Therefore $(u,\phi)=(0,0)$ which is a contradiction.

\section{Proof of Theorem \ref{th1}}\label{sec:ex}


Problem \eqref{P} is
variational and the related $C^1$ functional $F:H^1_0(\BR) \times
H^1_0(\BR) \rightarrow \mathbb{R}$ is given by
\[
F(u,\phi)
=\frac{1}{2}\int_{\BR} |\nabla u|^2
-\frac{\l}{2}\int_{\BR} u^2
-\frac{q}{5}\int_{\BR} |u|^5 \phi 
+\frac{1}{10}\int_{\BR} |\nabla \phi|^2.
\]

The functional $F$ is strongly indefinite. To avoid this indefiniteness, we apply the following reduction argument.\\
First of all we give the following result.
\begin{lemma}
\label{le:inv} For every $u\in H^1_0(\BR)$ there exists a unique
$\phi_u\in H^1_0(\BR)$ solution of
\[
\left\{
\begin{array}{ll}
    -\Delta \phi=q |u|^5
    &
    \hbox{in } \BR,\\
    \phi=0
    &
    \hbox{on } \partial \BR.
\end{array}
\right.
\]
Moreover, for any $u\in H^1_0(\BR)$, $\phi_u \ge 0$ and the map
\begin{equation*}
u \in H^1_0(\BR) \mapsto \phi_u \in H^1_0(\BR)
\end{equation*}
is continuously differentiable. Finally we have
\begin{equation}\label{eq:Ne2u}
\|\n\phi_u\|_2^2=q\int_{\BR} |u|^5 \phi_u
\end{equation}
and
\begin{equation}
\label{eq:essi}
\|\n\phi_u\|_2\leq \frac{q}{S^3} \|\n u\|_2^5
\end{equation}
where
\[
S=\inf_{v\in H^1_0(\BR)\setminus\{0\}}\frac{\|\n v\|_2^2}{\|v\|_6^2}.
\]
\end{lemma}
\begin{proof}
To prove the first part we can proceed reasoning as in \cite{BF}.\\ To show \eqref{eq:essi}, we argue in the following way. By applying H\"older and Sobolev inquality to \eqref{eq:Ne2u}, we get
\[
\|\n\phi_u\|_2^2 \le q \|\phi_u\|_6 \|u\|_6^5 \le \frac{q}{\sqrt{S}} \|\n\phi_u\|_2 \|u\|_6^5.
\]
Then
\[
\|\n\phi_u\|_2 \le \frac{q}{\sqrt{S}} \|u\|_6^5 \le \frac{q}{S^3} \| \n u \|_2^5.
\]
\end{proof}
So, using Lemma \ref{le:inv}, we can consider on $H^1_0(\BR)$ the $C^1$ one variable functional
        \begin{equation*}
        I(u):=F(u,\phi_u)=
        \frac{1}{2}\int_{\BR} |\nabla u|^2
        -\frac{\l}{2}\int_{\BR} u^2
        -\frac{1}{10}\int_{\BR} |\n \phi_u|^2 
        \end{equation*}

By standard variational arguments as those in \cite{BF}, the
following result can be easily proved.
\begin{proposition}
Let $(u,\phi)\in H^1_0(\BR)\times
H^1_0(\BR)$, then the following propositions are equivalent:
\begin{enumerate}[label=(\alph*), ref=\alph*]
\item $(u,\phi)$ is a critical point of functional $F$;
\item $u$ is a critical point of functional $I$ and
$\phi=\phi_u$.
\end{enumerate}
\end{proposition}

To find solutions of (\ref{P}), we look for critical points of $I$.

The functional $I$ satisfies the geometrical assumptions of the Mountain Pass Theorem (see \cite{AR}).\\
So, we set
\[
c=\inf_{\g \in \G} \max_{t\in [0,1]} I(\g(t)),
\]
where $\G=\left\{\g\in C([0,1],H^1_0(\BR)) \; \vline \; \g(0)=0, I(\g(1))<0\right\}$.\\
Now we proceed as follows:
\begin{enumerate}[label={\bf Step \arabic*:}, ref={Step \arabic*}]
\item \label{step1} we prove that there exists a nontrivial solution to the problem \eqref{P};
\item \label{step2} we show that such a solution is a ground state.
\end{enumerate}

\begin{remark}
Observe that standard elliptic arguments based on the maximum
principle work, so that we are allowed to assume that $u$ and $\phi_u$, solutions of \eqref{P}, are both positive.
\end{remark}



\noindent{\bf Proof of \ref{step1}:} {\it there exists a solution of \eqref{P}}.

Let $(u_n)_n$ be a Palais-Smale sequence at the mountain pass level $c$.
It is easy to verify that $(u_n)_n$ is bounded so, up to a
subsequence, we can suppose it is weakly convergent.

 Suppose by contradiction that $u_n\rightharpoonup 0$ in $H^1_0(\BR)$. Then $u_n\to 0$ in $L^2(\BR)$.\\
 Since $I(u_n)\to c$ and $\langle I'(u_n),u_n\rangle\to 0$ we have
    \begin{equation}\label{eq:one}
       \frac 1 2 \|\n u_n\|_2^2 -\frac 1 {10} \| \n \phi_{n}\|^2=c + o_n(1)
    \end{equation}
and
    \begin{equation}\label{eq:two}
        \|\n u_n\|_2^2 - \| \n \phi_{n}\|^2 = o_n(1)
    \end{equation}
where we have set $\phi_n=\phi_{u_n}$.
Combining \eqref{eq:one} and \eqref{eq:two} we have
    \begin{equation*}
        \|\n u_n\|_2^2=\frac 5 2 c + o_n(1)
    \end{equation*}
    and
    \[
    \|\n \phi_n\|_2 = \frac 5 2 c + o_n(1).
    \]
Then, since $(u_n,\phi_n)$ satisfies \eqref{eq:essi}, passing to the limit we get
    \begin{equation}\label{eq:cont}
        c \ge \frac 25\sqrt{\frac  {S^3} q}.
    \end{equation}
Now consider a fixed smooth function
$\varphi=\varphi(r)$ such that $\varphi(0)=1,$ $\varphi'(0)=0$ and $\varphi(R)=0.$ Following \cite[Lemma 1.3]{BN}, we set $r=|x|$ and
    \begin{equation*}
        u_\eps (r)=\frac {\varphi(r)}{(\eps + r^2)^{\frac 1 2}}.
    \end{equation*}
The following estimates can be found in \cite{BN}
    \begin{align*}
        \|\n u_\eps\|_2^2&=S\frac{K}{\eps^{\frac12}}+\o\int_0^R|\varphi'(r)|^2\,dr + O(\eps^{\frac12}),\\
        \|u_\eps\|_6^2&=\frac{K}{\eps^{\frac12}}+O(\eps^{\frac12}),\\
        \|u_\eps\|_2^2&=\o\int_0^R\varphi^2(r)\,dr+O(\eps^{\frac12}),
    \end{align*}
where $K$ is a positive constant and $\o$ is the area of the unitary sphere in $\RT.$

We are going to give an estimate of the value $c$. Observe that,
multiplying the second equation of \eqref{P} by $|u|$ and
integrating, we have that
    \begin{equation}\label{eq:ineq}
        q\|u\|_6^6 = \int_{\BR} (\n\phi_u|\n |u|)\le \frac 1{2} \|\n\phi_u\|_2^2 + \frac 12\|\n |u|\|_2^2.
    \end{equation}
So, if we introduce the new functional $J:H^1(\BR)\to\R$ defined in the following way
    $$J(u):= \frac {3}{5} \int_{\BR} |\n u|^2 - \frac \l2 \int_{\BR} u^2 - \frac q5\int_{\BR}|u|^6,$$
by \eqref{eq:ineq} we have that $I(u)\le J(u),$ for any $u\in
H^1_0(\BR),$ and $c\le \displaystyle\inf_{u\in H^1_0(\BR)
\setminus\{0\}}\sup_{t>0} J(tu).$

Now we compute $\sup_{t>0} J(tu_\eps)=J(t_\eps u_\eps),$ where
$t_\eps$ is the unique positive solution of the equation
    $$\frac d {dt} J(tu_\eps)=0.$$
Since
    $$\frac d {dt} J(tu_\eps)=\frac {6}5 t\int_{\BR} |\n u_\eps|^2 -  \l t \int_{\BR} u_\eps^2 -
    \frac 65t^5 q \int_{\BR}|u_\eps|^6,$$
we have that
$$
t_\eps= \frac 1 {\|u_\eps\|_6}\sqrt[4]{\frac{\frac {6}5 \|\n
u_\eps\|^2_2-\l\|u_\eps\|_2^2}{\frac 65 q\|u_\eps\|_6^2}}=\frac 1
{\|u_\eps\|_6}\sqrt[4]{\frac S q+A(\varphi) \eps^{\frac12} +
O(\eps)},
$$
where we have set
    $$A(\varphi)=\frac{\o}{q K} \int_0^R\left(|\varphi'(r)|^2-\frac 5
{6}\l\varphi^2(r)\right)\,dr.$$

Then
    \begin{align}\label{eq:est}
        \sup_{t>0} J(tu_\eps)&=J(t_\eps u_\eps)\nonumber\\
        &=\frac {3}{5} t_\eps^2 \int_{\BR} |\n u_\eps|^2 -
        \frac \l2 t_\eps^2 \int_{\BR} u_\eps^2 - \frac q5 t^6_\eps\int_{\BR}|u_\eps|^6\nonumber\\
        &=\frac 25 q \sqrt{\left(\frac S q+A(\varphi) \eps^{\frac12}+
        O(\eps)\right)^3}.
    \end{align}
Now, if we take $\varphi(r)=\cos(\frac{\pi r}{2R})$ as in \cite{BN}, we have that
    $$\int_0^R|\varphi'(r)|^2\,dr=\frac{\pi^2}{4R^2}\int_0^R\varphi^2(r)\,dr$$
and then, if $\l \in ]\frac 3{10} \l_1,\l_1[$, we deduce that
$A(\varphi)<0$. Taking $\eps$ sufficiently small, from
\eqref{eq:est} we conclude that $c<\frac 25\sqrt{\frac  {S^3} q},$
which contradicts \eqref{eq:cont}.

Then we have that $u_n\rightharpoonup u$ with
$u\in H^1_0(\BR) \setminus\{0\}$. We are going to prove that $u$ is a weak
solution of \eqref{P}.\\
As in \cite{ADL} it can be showed that $\phi_n\rightharpoonup\phi_u$
in $H^1_0(\BR).$ Now, set $\varphi$ a test function. Since
$I'(u_n)\to 0,$ we have that
    \begin{equation*}
       \langle I'(u_n),\varphi\rangle\to 0.
    \end{equation*}
On the other hand,
    \begin{multline*}
        \langle I'(u_n),\varphi\rangle= \int_{\BR} (\n u_n|\n\varphi) -\l\int_{\BR} u_n\varphi\\
        - q\int_{\BR}\phi_{n}|u_n|^3u_n\varphi\to \int_{\BR} (\n u|\n\varphi)-\l\int_{\BR} u\varphi- q\int_{\BR}\phi_{u}|u|^3u\varphi
    \end{multline*}
so we conclude that $(u,\phi_u)$ is a weak solution of \eqref{P}.

\medskip

\noindent{\bf Proof of \ref{step2}:}  {\it The solution found is a ground state}.\\
As in \ref{step1}, we consider a Palais-Smale sequence $(u_n)_n$ at level $c$. We have that $(u_n)_n$ weakly converges to a critical point $u$ of $I$.\\
To prove that such a critical point is a ground state we proceed as follows.\\
First of all we prove that
\[
I(u)\le c.
\]
Since $I(u_n)\to c$ and $\langle I'(u_n), u_n\rangle\to 0$, then
\[
I(u_n)= \frac 2 5 \int_{\BR} |\n u_n|^2 - \frac 2 5 \l \int_{\BR} u_n^2 + o_n(1) \to c.
\]
Moreover, being $(u,\phi_u)$ is a solution, we have
\[
\int_{\BR} |\n u|^2 - \l \int_{\BR} u^2 - q \int_{\BR} \phi_{u}|u|^5 = 0.
\]
Hence, by the lower semi-continuity of the $H^1_0$-norm and since $u_n \to u$ in $L^2(\BR)$,
\begin{align*}
I(u)
 = & \frac 2 5\left(\int_{\BR} |\n u|^2 - \l  \int_{\BR} u^2\right)\\
 \le & \frac 2 5 \left(\liminf_n \int_{\BR} |\n u_n|^2 - \l \lim_n \int_{\BR} u_n^2\right)\\
 = & \frac 2 5 \liminf_n  \left( \int_{\BR} |\n u_n|^2 - \l \int_{\BR} u_n^2 \right)\\
 = & c
\end{align*}
%
%
%
%
%
%
%
%
%
%
Finally, let $v$ be a nontrivial critical point of $I$.
Since the maximum of $I(tv)$ is achieved for $t=1$, then
\[
I(v)=\sup_{t>0} I(tv) \ge c \ge I(u).
\]

\end{document}